\documentclass[12pt,twoside]{amsart} 
 \title[Remarks on polarized endomorphisms]{Remarks on log Calabi--Yau  structure of varieties admitting polarized endomorphisms}

\author{Ama\"el Broustet}
\address{Universit{\'e} Lille 1, UMR CNRS 8524,
UFR de math{\'e}matiques, 59 655 Villeneuve d'Ascq CEDEX.}
\email{amael.broustet@math.univ-lille1.fr}

\author{Yoshinori  Gongyo}
\address{Graduate School of Mathematical Sciences, 
the University of Tokyo, 3-8-1 Komaba, Meguro-ku, Tokyo 153-8914, Japan.}
\email{gongyo@ms.u-tokyo.ac.jp}

\date{\today}


\newcommand{\Pic}[0]{{\operatorname{Pic}}}

\usepackage{pxfonts}

\usepackage{latexsym}
\usepackage{amsmath}
\usepackage{amssymb}
\usepackage{amsthm}
\usepackage{amscd}
\usepackage{enumerate} 
\usepackage{amssymb} % \mathfrak , \mathbb
\usepackage{mathrsfs}          % \mathscr  
\usepackage[all]{xy}
\usepackage{color}

\newtheorem{thm}{Theorem}[section]

\newtheorem{prop}[thm]{Proposition}

\newtheorem{lem}[thm]{Lemma}

\newtheorem{rem}[thm]{Remark}

\newtheorem{cor}[thm]{Corollary}

\newtheorem{conj}[thm]{Conjecture}

\newtheorem{ques}[thm]{Question}

\theoremstyle{definition}

\newtheorem{defi}[thm]{Definition}

\newtheorem{case}{Case}

\newtheorem*{ack}{Acknowledgments}

\begin{document}
\bibliographystyle{amsalpha+}
 
 \maketitle
 
 \begin{abstract}We discuss the Calabi--Yau type structure of normal projective surfaces  and Mori dream spaces admitting non-trivial polarized endomorphisms. 
 
 \end{abstract}
 
 \section{Introduction} Through this article, we work over the
 complex number field $\mathbb{C}$ except where otherwise stated. Recently  Yuchen Zhang studied the following question in his paper \cite{Yuchen}:
 \begin{ques}\label{yz-ask}Let $(X,0)$ be a normal isolated singularity. If $\phi : (X,0) \to (X,0)$ is a noninvertible finite endomorphism, does there exist a boundary $\Delta$ such that $(X, \Delta)$ is log canonical?
  \end{ques}

 The case of surfaces was studied by Wahl \cite{wahl} and Favre
 \cite{favre}. Defining the volume of a normal surface singularity,
 Wahl was able to give a positive answer to the question in the
 surface case. This invariant was generalized by Boucksom, de Fernex and Favre \cite{BdFF}
to isolated singularities of higher dimensional varieties. As a
consequence, they
obtained a positive answer to the question in the
$\mathbb{Q}$-Gorenstein case. Unfortunately, as shown by  Yuchen Zhang
in \cite{zhang2014}, their generalization cannot be used to treat the
non-$\mathbb{Q}$-Gorenstein case. Introducing a new volume refining
the one introduced in \cite{BdFF}, it is shown in
\cite{Yuchen} that  Question \ref{yz-ask} holds in the
non-$\mathbb{Q}$-Gorenstein case for \'etale in codimension one endomorphisms. 

 %[Add some historical comments on the above questions likes by Wahl, de Fernex, Boucksom, Favre, BH, etc.... ]

 In this article we study the following conjecture, which is a global version of Question \ref{yz-ask}:
 
 \begin{conj}\label{cy conj}Let $X$ be a normal projective varieties
   admitting a non-trivial polarized endomorphism. Then $X$ is of Calabi--Yau type.

\end{conj}

Compared with the study of Question \ref{yz-ask}, the study of
Conjecture \ref{cy conj} is closer to the  classification theory. Indeed Fujimoto and
Nakayama classify smooth compact complex surfaces admitting
non-trivial endomorphisms (which are not necessary polarized)
\cite{fn1}. For singular surfaces, however, the classification is more
complicated and not complete. %Nakayama has related results for this
%classification in his preprint \cite{nakayama-end}. %He classified the
%non-rational cases in dimension two \cite[Theorems\,1.3 and 1.4]{nakayama-end}.   
For higher dimension, Nakayama and Zhang study
the global structure of varieties admitting non-trivial polarized
endomorphisms from the viewpoint of the maximal rationally connected
fibrations \cite{nakayama-zhang10}. See also the recent paper
\cite{mz1}. And Zhang studies such a variety which is uniruled in more
detail in \cite{Zhang-uniruled}. For a more arithmetic point of view, see \cite{swzhang}. For positive characteristics, see \cite{nakayama-posi}. For higher dimensional varieties, see \cite{f1}, \cite{fs1}, \cite{nakayama-zhang10}, \cite{Zhang-uniruled}, and \cite{mz1}.

In this article, we obtain the following partial result toward Conjecture \ref{cy conj}:

\begin{thm}[Theorems \ref{cy mds} and \ref{thm:surface}
  ]\label{main-rem}Let $X$ be a Mori dream space or a normal
  projective  surface. If there exists a non-trivial polarized
  endomorphism on $X$, then there is an effective $\mathbb{Q}$-divisor
  $\Delta$ such that $K_X+\Delta \sim_{\mathbb{Q}}0$ and $(X, \Delta)$
  is log canonical, i.e. $X$ is of Calabi--Yau type (see Definition \ref{Fano pair}).
\end{thm}

We combine the techniques of \cite[Theorem\,1.4]{gost} and
\cite[Theorem\,1.1]{bh} for our proof
of Theorem \ref{main-rem} for Mori dream spaces. 

In the case of surfaces, %we need to discuss more than Nakayama's classification. 
we can check that Theorem \ref{main-rem} holds for smooth surfaces by
Fujimoto--Nakayama's classification \cite{f1} and
\cite[Theorem\,3]{nakayama-ruled}. But in general, since
endomorphisms can not be lifted on the minimal resolution, it does not seem
easy to reduce the singular case to the smooth one. Thus we need to argue
more directly and discuss separated cases according to the anti-Kodaira dimension, see Section \ref{surfaces}. 
%Nakayama seems to have a result related to Theorem \ref{thm:surface} in an updated version of his preprint \cite{nakayama-end}.
We also treat Conjecture \ref{cy conj} in the \'etale  in codimension one case, which is the global version of \cite{Yuchen}, in Section \ref{codim 1 etale}.  

\begin{ack} A.B is partially supported by ANR Labex CEMPI
  (ANR-11-LABX-0007-01), ANR projects CLASS (ANR-10-JCJC-0111) and MACK (ANR-10-BLAN-0104). Y.G is partially supported by the Grant-in-Aid for Scientific Research
  (KAKENHI No.\ 26707002). The authors express their gratitude to
  Professors Sebastian Boucksom, De Qi Zhang,  Keiji Ogusio, Yujiro
  Kawamata, Paolo Cascini, and  Mr. Sheng Meng for useful discussions and comments. Moreover the authors thank to the referee for pointing out serious and minor mistakes in the previous version.
\end{ack}

 \section{Preliminaries}

In this section, we collect definitions and lemmas needed for the proof of Theorem \ref{main-rem}.

\begin{defi}\label{def of pol end}
Let $X$ be a normal projective variety and $f: X \to X$ a surjective morphism.

We call $f$ {\it polarized} if there exists an ample Cartier divisor $H$ and a positive $q \in \mathbb{Z}$ such that $f^*H \sim_{\mathbb{Z}} qH$. We call it {\it non-trivial} if $f$ is not an isomorphism.
\end{defi}

In this paper, we use the following standard terminologies:  

\begin{defi}[{cf.~\cite[Definition 2.34]{komo}, \cite[Remark 4.2]{schsmith-logfano}}]\label{sing of pairs}
Let $X$ be a normal variety  and $\Delta$ be an  $\mathbb{Q}$-divisor on $X$ such that $K_X+\Delta$ is $\mathbb{Q}$-Cartier. 
Let $\pi: \widetilde{X} \to X$ be a birational morphism from a normal variety $\widetilde{X}$.
Then we can write 
$$K_{\widetilde{X}}=\pi^*(K_X+\Delta)+\sum_{E}a(E, X, \Delta) E,$$ 
where $E$ runs through all the distinct prime divisors on $\widetilde{X}$ and the $a(E, X, \Delta)$ are rational numbers. 
We say that the pair $(X, \Delta)$ is \textit{sub log canonical} (for short \textit{sub lc}) 
 if $a(E, X, \Delta) \ge -1$  for every prime divisor $E$ over $X$.  In particular, we just call \textit{log canonical} (for short \textit{lc}) if $\Delta$ is effective and $(X,\Delta)$ is sub lc.  We say that $(X, \Delta)$ is plt if  $a(E, X, \Delta) >-1$ for every {\it exceptional}  prime divisor $E$ over $X$ and $\Delta$ is effective.
If $\Delta=0$, we simply say that $X$ has log canonical singularities. 
\end{defi}

\begin{defi}[cf. {\cite[Lemma-Definition 2.6]{prokshok-mainII}}]\label{Fano pair}Let $X$ be a normal projective variety and $\Delta$ be an effective $\mathbb{Q}$-divisor on $X$ such that $K_X+\Delta$ is $\mathbb{Q}$-Cartier. We say that $(X,\Delta)$ is {\em log Calabi-Yau} if $K_X+\Delta \sim_{\mathbb{Q}}0$ and $(X, \Delta)$ is log canonical. 
We say that $X$ is of \textit{Calabi--Yau type} if 
there exists an effective $\mathbb{Q}$-divisor $\Delta$ on $X$ such that $(X, \Delta)$ is log Calabi-Yau.  
%\begin{enumerate}[(i)]
%\item We say that $(X,\Delta)$ is {\em log Fano} if $-(K_X+\Delta)$ is ample and $(X, \Delta)$ is klt. 
%We say that $X$ is of {\em Fano type} if there exists an effective $\mathbb Q$-divisor $\Delta$ on $X$ such that $(X,\Delta)$ is log Fano. 
%\item  We say that $(X,\Delta)$ is {\em log Calabi-Yau} if $K_X+\Delta \sim_{\mathbb{Q}}0$ and $(X, \Delta)$ is log canonical. 
%We say that $X$ is of \textit{Calabi--Yau type} if 
%there exists an effective $\mathbb{Q}$-divisor $\Delta$ on $X$ such that $(X, \Delta)$ is log Calabi-Yau.  
%\end{enumerate}
\end{defi}

 The following two results are used in the proof of Theorem \ref{main-rem};
 
 \begin{lem}[{\cite[Lemma 2.1]{nakayama-zhang10}}]\label{nz-lem}Let
   $X$ be an $n$-dimensional normal projective variety admitting an
   endomorphism $f$ such that there exists a nef and big Cartier
   divisor $H$ and positive real number $q$ such that $f^*H\equiv qH$,
   i.e. $f$ is quasi-polarized. Then $\mathrm{deg}\,f=q^{n}$ and $|\lambda|=q$ for  any eigenvalue $\lambda$ of  $f^*: N^1(X)_{\mathbb{R}} \to N^1(X)_{\mathbb{R}}$.  \end{lem}

The above lemma is obtained by a nice application of the Perron--Frobenius Theorem, see the proof of \cite[Lemma 2.1]{nakayama-zhang10}.

\begin{thm}[{\cite[Theorem 1.1]{bh}}]\label{singularities}Let $X$ be a normal projective varieties such that there exists a non trivial polarized endomorphism. If $X$ is $\mathbb{Q}$-Gorenstein, then it has log canonical singularities.

\end{thm}

 The above theorem is proved by using Odaka--Xu's lc modifications for pairs \cite[Theorem 1.1]{ox} and a careful study of lc centers in \cite{bh}.  
 
 And we have the same theorem for normal surfaces;
 
 \begin{thm}\label{singularities-char p}Let $X$ be a normal projective surface  such that there exists a non-trivial polarized  endomorphism on $X$. Then it has log canonical singularities. In particular, $X$ is $\mathbb{Q}$-Gorenstein. 

\end{thm}

\begin{proof}Apply \cite[Proposition\,3.1]{bh}  and \cite[Theorem (a) (b) (c), p. 626]{wahl}.

\end{proof}

\begin{lem}\label{lem-rational}Let $S$ be a normal projective rational surface.  If $H^i(S,\mathcal{O}_S)=0$ for $i=1,2$, then $S$ has rational singularities.

\end{lem}

\begin{proof}Let $f:W \to S$ be a resolution of singularities. By the Leray spectral sequence, the vanishing of $H^i(S,\mathcal{O}_S)$ for $i=1,2$ implies that  $H^i(S, f_*\mathcal{O}_W)=0$ for $i=1,2$. Then we have $H^1(W, \mathcal{O}_W)\simeq H^0(S, R^1f_*\mathcal{O}_W)$ by the following exact sequence:
$$0\to H^1(S, \mathcal{O}_S) \to H^1(W, \mathcal{O}_W) \to H^0(S, R^1f_*\mathcal{O}_W) \to H^2(S, f_*\mathcal{O}_W)\to \cdots.
$$

 Thus $H^0(S, R^1f_*\mathcal{O}_W)=0$. Now the support of $R^1f_*\mathcal{O}_W$ is a finite set. Thus we have $R^1f_*\mathcal{O}_W=0$.

\end{proof}

\begin{rem}\label{curve}Let $C$ be a normal projective curve admitting a  separated non-trivial  endomorphism. Then $(C, \frac{1}{q-1}R)$ is a Calabi--Yau pair, where $R$ is the ramification divisor and $q$ is degree of the endomorphism.

\end{rem}

 \section{Mori dream space case}We prove in this section the Conjecture \ref{main-rem} for Mori dream spaces.

\subsection{Mori dream spaces}
Mori Dream Spaces were first introduced by Hu and Keel \cite{hk}. 
\begin{defi}\label{Mori dream space} 
A normal projective variety $X$  is called a \textit{$\mathbb{Q}$-factorial Mori dream space} (or \textit{Mori dream space} for short) 
if $X$ satisfies the following three conditions:
\begin{enumerate}[(i)]
\item $X$ is $\mathbb{Q}$-factorial, $\Pic{(X)}$ is finitely generated,
and $\Pic{(X)}_{\mathbb{Q}} \simeq \mathrm{N}^1{(X)}_{\mathbb{Q}},$\label{fin_pic}
\item $\mathrm{Nef}{(X)}$ is the affine hull of finitely many semi-ample
line bundles, 
\item there exists a finite collection of small birational maps $f_i: X \dasharrow X_i$
such that each $X_i$ satisfies (i) and (ii), and that $\mathrm{Mov}{(X)}$ is the union
of the $f_i^*(\mathrm{Nef}{(X_i)})$.
\end{enumerate}
\end{defi}

\begin{rem}
Over the complex number field,
the finite generation of $\Pic{(X)}$ is equivalent to the condition
$\Pic{(X)}_{\mathbb{Q}} \simeq \mathrm{N}^1{(X)}_{\mathbb{Q}}$.
\end{rem}
On a Mori dream space, as its name suggests, we can run an MMP for any divisor.
\begin{prop}{ (\cite[Proposition 1.11]{hk})}\label{mmp on mds}
Let $X$ be a $\mathbb{Q}$-factorial Mori dream space. 
Then for any divisor $D$ on $X$, a $D$-MMP can be run and terminates. 
\end{prop}

Moreover the finiteness of nef cone implies the following lemmas:
 
 \begin{lem}\label{projective triviality}Let $X$ be a Mori dream space such that there exists a non trivial polarized endomorphism $f$. Then there exists $m \in \mathbb{N}$ and $\alpha \in \mathbb{R}$  such that $f^{m*} = \alpha  \cdot \mathrm{Id}$ on $\mathrm{NS}_{\mathbb{R}} (X)$. 
 
 \end{lem}

\begin{proof}Note that an endomorphism acts on the nef and pseudo effective cones. From the property of Mori dream spaces, the nef cone is rational polyhedral. Thus$f^m$ trivially acts on the set of extremal ray of the nef cone of $X$ for a  sufficiently large and divisible $m$. By Lemma \ref{nz-lem}, $f^{m*} = \alpha \cdot \mathrm{Id}$ on the nef cone, where $\alpha$ is $q$ in  Lemma \ref{nz-lem}. The nef cone contains some basis of  $\mathrm{NS}_{\mathbb{Q}} (X)$. Thus Lemma \ref{projective triviality} follows.
\end{proof}

\begin{lem}\label{induced mor-1}Let $X$ be a projective scheme and $H$ a semi-ample $\mathbb{Q}$-divisor. Suppose that there exists an endomorphism $f:X \to X$ such that $f^*H \sim_{\mathbb{Q}} qH$ for some $q \in \mathbb{Q}$.  Let $g:X \to Y$ be the algebraic fibre space induced by $H$. Then $f$ induces  a polarized  endomorphism  $f_Y$ of $Y$. 

\end{lem}

\begin{proof}Take the Stein factorization $g':X \to Y'$ of $g \circ
  f:X \to Y$. Then we have an isomorphism $g \simeq g'$ since $g$ and
  $g'$ are algebraic fibre spaces and we see that  the same curves
  are contracted by looking at the intersection numbers. Let $h:Y \simeq Y' \to Y$ be the induced map. We see that $h$ gives a polarization. Indeed by definition we have a $\mathbb{Q}$-ample divisor $H_Y$ such that $g^*H_Y\sim_{\mathbb{Q}}H$. Thus we have also  $h^*H_Y\sim_{\mathbb{Q}}qH_Y$ by the projection formula.

\end{proof}

\begin{lem}\label{induced mor-2}Let $\pi:X \to S$ be an algebraic fiber space  of normal projective varieties and $H$ a  Cartier divisor such that the section ring $R_\pi(X,H)$ is a finitely generated  $\mathcal{O}_S$-algebra, where 
$$\displaystyle R_\pi(X,H):=\bigoplus_{m \in \mathbb{Z}_{\geq 0}} \pi_*\mathcal{O}_X(mH).$$ Suppose that there exists endomorphisms $f:X \to X$ and  $f_S:S\to S$ such that the diagram
$$ \begin{CD}
X @>f>> X\\
@VV\pi V @VV\pi V\\
S @>f_S>> S
\end{CD}$$

commutes, $f^*H \sim_{\mathbb{Z}, \pi} qH$ for some $q \in \mathbb{Z}$, and $f_{S}^*H_S \sim_{\mathbb{Z}, \pi} qH_S$ for some ample Cartier divisor $H_S$ on S (i.e. $f_S$ is polarized).  Let $g:X \dashrightarrow Y$ be a dominant rational map to $Y=\mathrm{Proj}_{S}\,R_{\pi}(X,H)$. Then $f$ induces  a polarized  endomorphism  $f_Y$ of $Y$ over $S$ such that the diagram
$$ \begin{CD}
Y @>f_Y>> Y\\
@VV \pi_Y V @VV \pi_Y V\\
S @>f_S>> S
\end{CD}$$

commutes. 

\end{lem}

\begin{proof}%Let $m$ and $k$ be positive integers such that $kq$ is integer and $mH \sim_{\mathbb{Z}} kq H$. By the condition of polarization, we have the induced graded ring homomorphism 
We have $$f^*:R_{\pi}(X,H) \to R_{\pi}(X,qH).$$ Then we obtain a morphism $f_Y:Y \to Y$ over $S$ satisfying  the diagram
$$ \begin{CD}
Y @>f_Y>> Y\\
@VVV @VVV\\
S @>f_S>> S
\end{CD}$$

commutes.
We show that  $f_Y$ is polarized.  Indeed let $H'$ be  the tautological bundle of  $\mathrm{Proj}_S\,R_{\pi}(X,mH).$ Then we see that $f_Y^*H' \sim_{\mathbb{Z}} qH'$. Thus  $H_Y:= H'+l \pi_Y^*H_S$ is ample on $Y$ for $l \gg 0. $ Since $f_Y^*H_Y \sim_{\mathbb{Z}} qH_Y$, $f_Y$ is also polarized. 
\end{proof}

 \begin{lem}\label{non van mds}Let $X$ be a Mori dream space admitting a non-trivial polarized endomorphism. Then $\kappa(-K_X) \geq 0$ if $K_X$ is $\mathbb{Q}$-Cartier.
 \end{lem}

\begin{proof}By Lemma \ref{projective triviality}, we can see that $f^{m*}K_X\sim_{\mathbb{Q}}lK_X$ for some integer $l \geq 2$ and $m \geq 1$. 
On the other hand, by the ramification formula, we have
$$K_X=f^{m*}K_X+R,
$$
where $R$ is the ramification divisor, in particular, it is effective.  Thus we conclude that $R\sim_{\mathbb{Q}} (1-l)K_X.$
\end{proof}

Finally we prove that  Conjecture \ref{main-rem} holds for Mori dream spaces.

\begin{thm}\label{cy mds}Let $X$ be a Mori dream space such that there exists a non--trivial polarized endomorphism. Then $X$ is of Calabi--Yau type.

\end{thm}

\begin{proof}Running a $(-K_X)$-MMP, we have an anti-minimal model $Y$
  of $X$. By Lemmas \ref{projective triviality}, \ref{induced mor-1} and \ref{induced mor-2}, $Y$
  has a non-trivial polarized endomorphism. Thus $Y$ has log canonical
  singularities by Theorem \ref{singularities} and $-K_Y$ is semi-ample by Lemma \ref{non van mds}. By the arguments of \cite[Theorem 1.5]{gost}, $X$ is of Calabi--Yau type. 

\end{proof}

\section{ \'Etale  in codimension one case}\label{codim 1 etale}
We see that Conjecture \ref{cy conj} is true when $f$ is \'etale  in codimension one and $X$ has $\mathbb{Q}$-Gorenstein singularities. Note that \cite[Theorem 1.1]{Yuchen} implies that the following Proposition holds for smooth projective varieties:
\begin{prop}\label{etale in codim one}
Let $X$ be a normal $\mathbb{Q}$-Gorenstein projective variety such that there exists a non-trivial endomorphism with  \'etale  in codimension one such that $f^*H=qH$ for some nef and big $\mathbb{R}$-Cartier divisor and some $q \in \mathbb{R}$. Then $X$ is  log canonical and $K_X\sim_{\mathbb{Q}}0 $ 
\end{prop}

\begin{proof}Since $f$ is \'etale  in codimension one, we have 
$$K_X=f^*K_X.
$$
Thus we have $K_X \equiv 0$ by Lemma \ref{nz-lem} and that $q>1$. On the other hand, log-canonicity follows from Theorem \ref{singularities}. Moreover we have 
$$K_X\sim_{\mathbb{Q}}0 $$ by \cite[Theorem 1.2]{gongyo-abund}, \cite[Theorem 7]{kawamata-abunnuzero}, or \cite[Theorem 1]{ckp-num}.
\end{proof}

\section{Surfaces}\label{surfaces}

In this Section, we give a proof of Conjecture \ref{cy conj} for surfaces:

\begin{thm}\label{thm:surface}Let $X$ be a normal projective surface such that there exists a non trivial polarized endomorphism. Then $X$ is of Calabi--Yau type.
\end{thm}

Before starting a proof, we discuss a classification of surfaces
admitting non-trivial endomorphisms after Fujimoto and
Nakayama. Actually those surfaces are classified up to regular
equivalence when they are smooth.

%But we still have a nice partial classification by Nakayama:

%\begin{thm}[{\cite[Theorems 1.3 and 1.4]{nakayama-end}}]\label{nakayama
   % classification}Let $S$ be a compact complex Moishezon surface
  %admitting a non-trivial endomorphism. Then $S$ is projective. Moreover $S$ has non-rational singularities or $S$ is not rational if and only if  $S$ is one of the followings:

%\begin{itemize} 
 %\item[(1)] There is a finite Galois covering $A \to S$ \'etale in codimension one from an abelian surface $A$.
 %\item[(2)] $S$ is a projective cone over an elliptic curve.
 %\item[(3)]$S$ is a $\mathbb{P}^1$-bundle over an elliptic curve.
 %\item[(4)] There is a finite Galois covering $\mathbb{P}^1 \times T \to S$ \'etale in codimension one for a smooth projective curve $T$ of genus greater than one.
 %\item[(5)] There is a finite Galois covering $C \times T \to S$ \'etale in codimension one for an elliptic curve $C$ and  a smooth projective curve $T$ of genus greater than one.
 %\end{itemize} 

%\end{thm}

%In the above list, (1) and (2) are obviously of Calabi Yau type (indeed both cases have non-trivial polarized endomorphisms) and (4) and (5) never have polarized endomorphisms.\footnote{Because if these $S$ has polarized endomorphism $f$, $f^2$ is acting scholar on $N^1(S)$ since the picard number $2$. Thus $\kappa(-K_S)$ is non-negative. This is a contradiction!   }  Thus we discuss (3). Indeed, in Case (3), we see the following:

We begin by two results about non-rational smooth surfaces.

\begin{prop}[{\cite[Proposition 2.3.1]{swzhang}}]\label{bundle over elliptic curve}Let $\mathbb{P}_C(E)$ be a $\mathbb{P}^1$-bundle over an elliptic curve $C$. Then $\mathbb{P}_C(E)$ has a non-trivial polarized endomorphism if and only if $E=\mathcal{O}_C \oplus M$ with $M$ semi-ample or $E$ is indecomposable of odd degree. 

\end{prop}

\begin{proof}See \cite[Proposition 2.3.1]{swzhang}.

\end{proof}

In particular, in the above two cases, the log anti-canonical divisor with some boundary  is semi-ample:

\begin{prop}[cf. {\cite[Proposition\,1.6]{bp}}]\label{complement}Let $S=\mathbb{P}_C(E)$ be a $\mathbb{P}^1$-bundle over an elliptic curve $C$. If  $S$ has a non-trivial polarized endomorphism, then $-(K_S+Z)$ is semi-ample for some (possibly zero) divisor  $Z$  such that $(S,Z)$ is plt.
\end{prop}

\begin{proof}By Proposition \ref{bundle over elliptic curve}, we see
  $E =\mathcal{O}_C \oplus M$ with $M$ semi-ample or $E$ is
  indecomposable of odd degree.  First we show Proposition
  \ref{complement} for the indecomposable $E$. Then $E$ is a semi-stable
  vector bundle (cf. \cite{atiyah}, \cite[V., Theorem 2.15 and
  Exercise 2.8]{hartshorne}). Thus $-K_S$ is nef by \cite{ns unitary}
  or \cite[Theorem 3.1]{Mi1} and never big nor numerically
  trivial. Thus $-K_S$ generates the extremal ray of the nef cone
  which is deferent from the fibre of $S \to C$. On the other hand,
  the linear system $|-2K_S|$ is a pencil by the Riemann--Roch
  theorem. Thus $S$ has an elliptic fibration and $-K_S$ defines this
  fibration. Thus $-K_S$ is semi-ample. In the case of   $E
  =\mathcal{O}_C \oplus M$ with $M$ ample, let  $Z$ be the negative
  section $Z^2=-\mathrm{deg}\,M$. Then we can see that $-(K_S+Z)$ is nef and big, and $(S, Z)$ is plt. Moreover $-(K_S+Z)|_Z=-K_Z$ is trivial. Thus $-(K_S+Z)$ is semi-ample by the base point free theorem \cite[Theorem 1.1]{fujino-kawamata}. In the last part of    $E =\mathcal{O}_C \oplus M$ with $M$ torsion, take  the \'etale morphism $C'\to C$ corresponding to $M$. Then by the base change  of $S$  by $C'\to C$ is isomorphic to $\mathbb{P}^1 \times C' $. Moreover  $\mathbb{P}^1 \times C'  \to S$ is also \'etale. Thus $-K_S$ is semi-ample since so is $-K_{\mathbb{P}^1 \times C' }$.

\end{proof}

 Thus we show the following by using Fujimoto--Nakayama's classification.
 
 \begin{cor}\label{non-rational case}
Let  $S$ be a smooth  projective surface admitting a non-trivial polarized endomorphism. Then $S$ is of Calabi--Yau type.
 \end{cor}
 
 \begin{proof}By the Bertini's theorem and \cite[Proposition
   2.3.1]{swzhang}, Proposition \ref{complement}  implies Corollary
   \ref{non-rational case} for non-rational surfaces. On the other
   hand, a smooth  rational projective surface admitting a non-trivial
   endomorphism is toric (see \cite[Theorem 3]{nakayama-ruled}), thus
   a Mori dream space.
 
 \end{proof}
 
%Combining this results with Nakayama--Fujimoto's classification
%(cf. \cite[Proposition 2.3.1]{swzhang}), we can check that every {\it smooth } projective surface with non-trivial polarized endomorphism is of Calabi--Yau type. 

  However it seems difficult to have such a classification for
  singular surfaces. The main obstruction  to this classification  for
  singular surfaces seems to be the non-liftability of endomorphisms
  to minimal resolutions. Thus we can not reduce the problem to smooth cases. 
  
 We start the proof of Theorem \ref{thm:surface}. First we show Theorem  \ref{thm:surface} for rational surfaces with rational singularities:
 
 \begin{thm}\label{thm:surface-rational-rational}Let $S$ be a normal projective rational surface with only rational singularities such that there exists a non-trivial polarized endomorphism. Then $S$ is of Calabi--Yau type.
\end{thm}

\begin{proof}Since $S$ has rational singularities $S$, is
  $\mathbb{Q}$-factorial. Moreover we can use the following lemma due to Zhang:

\begin{lem}[{\cite[Theorem 2.7]{Zhang-uniruled}}]\label{zhang's surface lem}Let $S$ be a projective  $\mathbb{Q}$-Gorenstein surface admitting an endomorphism $f$ such that $f^*H\equiv qH$ for some $q>1$ and a big line bundle $H$. If  $K_S \not\equiv 0$,  then there exist $m \in \mathbb{N}$ and $\alpha \in \mathbb{Q}$  such that $f^{m*} = \alpha  \cdot \mathrm{Id}$ on $\mathrm{NS}_{\mathbb{Q}} (S)$.

\end{lem}
Assuming $K_S \not\equiv 0$, and since the result does not depend of
the multiple of $f$ considered, we may assume that $f^{*} = q  \cdot \mathrm{Id}$
on $\mathrm{NS}_{\mathbb{Q}} (S)$. By the same proof as in Lemma \ref{non
  van mds}, we have $\kappa(-K_S) \geq 0$. We argue case by case
according to the value of  $\kappa(-K_S)$.

\begin{case}When $\kappa(-K_S)=2$.

\end{case}
In this case, $-K_S$ is big. Take the minimal resolution  $\pi:T \to S$.  We have 
$$\pi^*K_S=K_T+E,
$$
where $E$ is an $\pi$-exceptional divisor. Thus $-K_T$ is big. By
\cite[1.11 Proposition]{hk}(cf. \cite{o} for a more general statement)
and \cite{tvav}, $S$ is a Mori dream space. Thus $S$ is of Calabi--Yau type by Theorem \ref{cy mds}.

%\begin{case}\label{kappa=0}When $\kappa(-K_S)=0$.

%\end{case}Take the unique effective $\mathbb{Q}$-divisor  $D$ such that $D \sim_{\mathbb{Q}} -K_S$.
%By the ramification formula, we have 

%$$ K_S=f^*K_S+R,
%$$
%where $R$ is the ramification divisor. Thus $R=(q-1)D$ since $\kappa(-K_S)=0$ and $S$ is a rational surface with rational singularities. Let $B$ be the branch divisor of $f$. We see that $\mathrm{Supp}\,R=\mathrm{Supp}\,f^*B$ by definition. Thus we have
 %$$\mathrm{Supp}\,R=\mathrm{Supp}\,B=\mathrm{Supp}\,D$$
%since $\kappa(R)=\kappa(B)=0$ and $f^* B=qB$.
%Then we have 
%$$ K_S+B=f^*(K_S+B)
%$$
%by the ramification formula. By Lemma \ref{nz-lem}, we see that $K_S+B\sim_{\mathbb{Q}}0$. Moreover $(S,B)$ is log canonical since $f$ can be lifted on the log canonical modification of $(S,B)$ by \cite[Lemma\,2.11]{bh}. 

\begin{case}When $\kappa(-K_S) \leq 1$.

\end{case}
Take the Zariski decomposition $-K_S\sim_{\mathbb{Q}}P+N$. Now $P$ is
nef and $\kappa(P)=\nu(P)=1$ or $P \sim_{\mathbb{Q}} 0$. In particular  $P$ is semi-ample. By Lemma 
\ref{zhang's surface lem} and $\kappa(N)=0$, $f^*N=qN$. Thus after replacing $f$ with a power  of $f$, we may assume that every component of $N$ is a totally invariant divisor of  $f$. Note that $N \leq \frac{1}{q-1}R$.  Let  $N_{red}$ be the reduced sum of prime divisors in $N$. 
Then $(S,N_{red})$ is log canonical  by \cite[Corollary\,3.3]{bh}. Since $N \leq \frac{1}{q-1}R$ and the components are total invariant, we have $N \leq N_{red}$.

Therefore we see that  a pair $(X, N)$ is log canonical.   Moreover since $P$ is semi-ample, we see that $S$ is of Calabi--Yau type.

\end{proof}

%In the above proof, we use the following lemma

%\begin{lem}[{cf. \cite[Theorem 2.7]{Zhang-uniruled}}]\label{zhang's surface lem finite}Let $S$ be a projective  surface over the algebraic closure of a finite field, admitting an separated endomorphism $f$ such that $f^*H\equiv qH$ for some $q>1$ and a big line bundle $H$. If  $K_S$ is not pseudo-effective,  then there exist $m \in \mathbb{N}$ and $\alpha \in \mathbb{Q}$  such that $f^{m*} = \alpha  \cdot \mathrm{Id}$ on $\mathrm{NS}_{\mathbb{Q}} (S)$. 

%\end{lem}

Finally we treat the cases of singular non-rational surfaces and singularities and give a proof of Theorem \ref{thm:surface}.
\begin{proof}[Proof of Theorem \ref{thm:surface}]By Theorem
  \ref{singularities-char p},  $K_S$ is  $\mathbb{Q}$-Cartier. We may assume that $K_S$ is not pseudo-effective by Lemma \ref{zhang's surface lem}. Indeed by the same proof as in Lemma \ref{non van mds}, we have $\kappa(-K_S) \geq 0$. Run an MMP for $S$ by \cite[Theorem 1.1]{fujino-surface}. Thus we have 
$$S\to S' \to C,$$
where $\varphi:S\to S'$ is a composition of divisorial contractions and $\pi:S' \to C$ is a Mori fiber space. By  Lemma \ref{zhang's surface lem}, we may assume that this is $f$-equivariant, i.e. there exist polarized endomorphisms $f':S'\to S'$ and   $g:C\to C$ commuting with $\varphi$ and $\pi$. Let $R$ be the ramification divisor of $f$. Then the strict transformation  $R'$of $R$ on $S'$ is the ramification divisor of $R'$. And  by Lemma \ref{zhang's surface lem}, we see that
\begin{eqnarray}\label{(1)}
\varphi^*(K_{S'}+\frac{1}{q-1}R')=K_{S}+\frac{1}{q-1}R \sim_{\mathbb{Q}}0.
\end{eqnarray}

\begin{case}When $\mathrm{dim}\,C=1$.
\end{case}
 In this case, $C$ is  a smooth rational curve or  an elliptic curve. In the case of rational curves, $S$ is a rational surface. And also we have $R^i\pi_*\mathcal{O}_{S'}= 0$ for $i>0$ by relative Kodaira vanishing theorem for log canonical surfaces (cf. \cite[Theorem 8.1]{fujino-funda}). Now since $C$ is smooth and rational, we have $H^i(S', \mathcal{O}_{S'})=0 $ for $i=1,2$ by the Leray spectral sequence. We have  $H^i(S, \mathcal{O}_{S})=0 $ for $i=1,2$ by again using the Leray spectral sequence.Then $S$ has rational singularities by Lemma
\ref{lem-rational}. On the other hand, in the case of elliptic curves, we apply  the canonical bundle formulas for a pair $(S', \frac{1}{q-1}R')$, after  replacing $f$ by its positive power. Note that  $(S', \frac{1}{q-1}R')$ is log canonical over the generic point of $\pi$ by using \cite[Theorem 5.1]{fakh}. Indeed by \cite[Theorem 5.1]{fakh} we can find some general point $y \in C $  such that $F_y:=\pi^{*}y$ is reduced and smooth, and $f$ induces an endomorphism $f_y: F_y \to F_y$. We get by  adjunction  formula 
$$K_{F_y}=f_y^*(K_{F_y})+ R'_{|F_y}
$$
as the ramification formula for $f_y$. Thus the coefficients of $\frac{1}{q-1}R'_{|F_y}$ are less than or equal to one by Remark \ref{curve}. Thus $(S', \frac{1}{q-1}R')$ is log canonical over the generic point of $\pi$.

 Let 
$$\displaystyle B=\sum_{P: \text{Weil divisor on $Y$}}(1-t_P)P,$$ 
where 
$$t_P:=\sup\{t \in \mathbb{R}| \,(X,\Delta+t\pi^*P) \text{ is sub lc
  over $P$}\}.$$ Then by \cite[Theorem 4.1.1 in the arXiv
version]{fujino-higher direct}  we have a pseudo-effective divisor $M$ such that 

$$K_{S'}+\frac{1}{q-1}R'=\pi^*(K_C+B+M).
$$ 
Note that $B$ is effective. Now since $K_{S'}+\frac{1}{q-1}R' \equiv 0$ and $K_C \sim 0$ we have $B=0 $ and $M \equiv 0$. So we see that $(S', \frac{1}{q-1}R')$ is log canonical by
the definition of $B$ and $K_S'+\frac{1}{q-1}R' \equiv 0$.  Thus we also see $(S, \frac{1}{q-1}R)$ is log Calabi--Yau pair.
 
 %By Lemma \ref{lc trivial} and Remark \ref{curve}, we see that $(S', \frac{1}{q-1}R')$ is a log Calabi--Yau pair. This implies that so is $(S, \frac{1}{q-1}R)$ by (\ref{(1)}).

\begin{case}When $\mathrm{dim}\,C=0$.
\end{case}
When $S$ is not rational, we see that $S'$ is isomorphic to the cone of an elliptic curve $E$ by \cite[Corollary 5.4.4]{P}. And also we may assume that $f^{-1}O=O$, where $O$ is the vertex of this cone by \cite[Lemma\,2.10]{bh}. 

%We take the blow-up $\psi:\tilde{S'}=\mathbb{P}(\mathcal{O}\oplus L) \to S'$ at $O$, where $L$ is some very ample line bundle on $E$ and then we induce the morphism $W \to \tilde{S'}$ by taking the blow up along the base ideal $\mathfrak{b}$ of rational map $f \circ \psi^{-1}$. This is a resolution of indeterminate locus of $f \circ \psi^{-1}$.   And we have the natural map $\tilde{S'}  \to W$ by the structure algebra since $\mathfrak{b} \subseteq I_O$, where $I_O$ is the coherent ideal sheaf of $O$. So  we have $\tilde{f'}:\tilde{S'} \to \tilde{S'}$ which is a lifting of $f'$. Since $\rho(\tilde{S'} )=2$ and $\tilde{f'}$ is quasi-polarized, we may assume that $\tilde{f'}^*$ is a scalar map on  $NS_{\mathbb{R}}(\tilde{S'} )$. Thus  $(\tilde{S'}, \frac{1}{q-1}R_{\tilde{f'}})$ is Calabi--Yau pair by Lemma \ref{lc trivial}. Thus so is   $(S', \frac{1}{q-1}R')$. Thus  $(S, \frac{1}{q-1}R)$ is also log Calabi--Yau by (\ref{(1)}).
This induces a polarized endomorphism of the elliptic curve $E$  by the fact that any dominant rational map of varieties induces a morphism between their Albanese varieties  (cf. \cite[Corollary 1.4]{mz1}). 
If $\varphi$ is not an isomorphism, then there exists a point $P \not =O
\in S'$ such that $f^{-1}P=P$. Thus  the endomorphism of $E$ is
ramified. But this is a contradiction.  Thus $\varphi$ is an isomorphism. Thus $S$ is isomorphic to the cone of an elliptic curve. In particular, it is of Calabi--Yau type. 

In the case of a rational surface $S$, the surface $S'$ is lc and
$-K_{S'}$ is ample. Then $H^i(S',\mathcal{O}_{S'})=0$ for $i=1,2$  (cf. \cite[Theorem 8.1]{fujino-funda}). By
the Leray spectral sequence, $H^i(S,\mathcal{O}_S)=0$ for
$i=1,2$. Then $S$ has rational singularities by Lemma
\ref{lem-rational}. We already have shown Theorem \ref{thm:surface} in the case of rational surfaces with rational singularities in Theorem \ref{thm:surface-rational-rational}. 
\end{proof}

\end{document}